\documentclass[12pt,reqno]{amsart}
\usepackage[notref,notcite]{}
\usepackage{amssymb}
\usepackage{amsmath}

\newtheorem{theorem}{Theorem}[section]
\newtheorem{lemma}[theorem]{Lemma}
\newtheorem{proposition}[theorem]{Proposition}

\newtheorem{definition}[theorem]{Definition}

\newtheorem{problem}[theorem]{Problem}
\newtheorem{conjecture}[theorem]{Conjecture}

\numberwithin{equation}{section}

\newcommand{\cout}[1]{}

\begin{document}
\title[Zeta-invariants of the Steklov spectrum]{Zeta-invariants of the Steklov spectrum\\ for a planar domain}

\author{Evgeny Malkovich and Vladimir Sharafutdinov}
\thanks{The first author was supported by Russian State Support of
Researches, Grant 14.B25.31.0029.\\
\indent The work was started by the second author when he stayed at Institut Mittag-Leffler in January -- March 2013 in the scope of the program ``Inverse Problems''. The second author is grateful to the institute for the support and hospitality.}
\address{Sobolev Institute of Mathematics and Novosibirsk State University, Russia}
\email{sharaf@math.nsc.ru}
\email{malkovich@math.nsc.ru}

\begin{abstract}
The classical inverse problem of recovering a simply connected smooth planar domain from the Steklov spectrum \cite{E} is equivalent to the problem of recovering, up to a conformal equivalence, a positive function $a\in C^\infty({\mathbb S})$ on the unit circle ${\mathbb S}=\{e^{i\theta}\}$ from the eigenvalue spectrum of the operator $a\Lambda_e$, where $\Lambda_e=(-d^2/d\theta^2)^{1/2}$. We introduce $2k$-forms $Z_k(a)\ (k=1,2,\dots)$ in Fourier coefficients of the function $a$ which are called zeta-invariants. They are uniquely determined by the eigenvalue spectrum of $a\Lambda_e$.
We study some properties of $Z_k(a)$, in particular, their invariance under the conformal group. Some open questions on zeta-invariants are posed at the end of the paper.
\end{abstract}

\maketitle

\section{Introduction. Three forms of an inverse problem\\ for the Steklov spectrum}

Let ${\mathbb D}=\{(x,y)\mid x^2+y^2\leq 1\}\subset{\mathbb R}^2={\mathbb C}$ be the unit disc and ${\mathbb S}=\partial {\mathbb D}=\{e^{i\theta}\mid\theta\in{\mathbb R}\}$, the unit circle. We introduce the first order pseudodifferential operator
\begin{equation}
\Lambda_e=\sqrt{-d^2/d\theta^2}:C^\infty({\mathbb S})\rightarrow C^\infty({\mathbb S}).
                                               \label{1.1}
\end{equation}
Equivalently, the operator is defined by $\Lambda_e e^{in\theta}=|n|\,e^{in\theta}$ on elements of the trigonometric basis. For a reason explained below, $\Lambda_e$ is called the {\it Dirichlet-to-Neumann operator of the Euclidean metric} $e$ (DN-operator briefly). The eigenvalue spectrum of the operator is
$$
\mbox{Sp}(\Lambda_e)=\{0,1,1,2,2,\dots\},
$$
where each eigenvalue is repeated according to its multiplicity.

For a positive function $a\in C^\infty({\mathbb S})$, the operator $a\Lambda_e$ has also a discrete eigenvalue spectrum
$$
\mbox{Sp}(a\Lambda_e)=\{0=\lambda_0<\lambda_1\leq\lambda_2\leq\dots\}
$$
that will be called the {\it Steklov spectrum} of the operator $a\Lambda_e$. In the present article, we discuss the question: to what extent is a function $0<a\in C^\infty({\mathbb S})$ determined by the Steklov spectrum $\mbox{Sp}(a\Lambda_e)$? The problem has a natural gauge group caused by conformal and anticonformal transformations of the disc ${\mathbb D}$. Let us introduce the corresponding definition.

For a smooth map $\varphi:{\mathbb S}\rightarrow{\mathbb S}$, the derivative $d\varphi/d\theta\in C^\infty({\mathbb S})$ is defined by $\varphi^*(d\theta)=(d\varphi/d\theta)\,d\theta$.

\begin{definition} \label{D1.1}
Two functions $a,b\in C^\infty({\mathbb S})$ are said to be conformally equivalent if there exists
a conformal or anticonformal transformation $\Phi:{\mathbb D}\rightarrow {\mathbb D}$ such that
\begin{equation}
b=a\circ\varphi\left|d\varphi/d\theta\right|^{-1},\quad\mbox{where}\quad\varphi=\Phi|_{\mathbb S}.
                                               \label{1.2}
\end{equation}
\end{definition}

If functions $a$ and $b$ do not vanish, equation (\ref{1.2}) can also be written in the form
$$
d\theta/b(\theta)=\pm\varphi^*\left(d\theta/a(\theta)\right).
$$

{\bf Remark.}
We emphasize the following difference between this definition and the corres\-pon\-ding definition of \cite[Section 3]{JS}: two positive functions $a$ and $b$ are conformally equivalent in our sense if and only if the functions $1/a$ and $1/b$ are $e$-conformally equivalent in the sense of \cite{JS}. We have changed the definition according to our wish to simplify the notation $a^{-1}\Lambda_e$ to $a\Lambda_e$.
Formally speaking, the operator $a\Lambda_e$ is well defined for an arbitrary (complex-valued) function $a\in C^\infty({\mathbb S})$ and some of our results make sense in this generality; although the eigenvalue spectrum $\mbox{Sp}(a\Lambda_e)$ will be discussed only in the case of a positive function $a$.

As can be easily proved, $\mbox{Sp}(a\Lambda_e)=\mbox{Sp}(b\Lambda_e)$ for conformally equivalent positive functions $a,b\in C^\infty({\mathbb S})$, see \cite{JS}. The converse statement is still open.

\begin{conjecture} \label {C1.1}
For two positive functions $a,b\in C^\infty({\mathbb S})$, the equality
\begin{equation}
\mbox{\rm Sp}(a\Lambda_e)=\mbox{\rm Sp}(b\Lambda_e)
                                               \label{1.3}
\end{equation}
holds if and only if these functions are conformally equivalent.
\end{conjecture}

To be honest, we are not optimistic about the validity of the conjecture in the general case. Nevertheless, there are many versions of the problem which are worth of studying even if the answer is "no" in the general case. For example, we can ask: how many positive functions $a\in C^\infty({\mathbb S})$ satisfy (\ref{1.3}) for a given $0<b\in C^\infty({\mathbb S})$? We believe that, for a generic $b$, such a function $a$ is unique up to the conformal equivalence.

\bigskip

There are two other equivalent forms of the same problem. We discuss them very briefly here. All details are presented in \cite{JS}.

Let $\Omega\subset{\mathbb R}^2$ be a simply connected domain bounded by a smooth closed curve $\partial\Omega$. The {\it Steklov spectrum} $\mbox{Sp}(\Omega)$ of the domain consists of $\lambda\in{\mathbb R}$ such that the boundary value problem
$$
\Delta u=0\quad\mbox{in}\quad\Omega,\quad\left.\partial u/\partial\nu\right|_{\partial\Omega}=-\lambda u|_{\partial\Omega}
$$
has a non-trivial solution. Here $\nu$ is the unit outward normal to the boundary. As well known, the spectrum $\mbox{Sp}(\Omega)$ is discrete and non-negative. The classical inverse problem sounds as follows: to what extent is a simply connected smooth bounded domain $\Omega\subset{\mathbb R}^2$ determined by its Steklov spectrum? Here, the natural conjecture is as follows:

\begin{conjecture} \label {C1.2}
A simply connected smooth bounded domain $\Omega\subset{\mathbb R}^2$ is determined by its Steklov spectrum uniquely up to an isometry of ${\mathbb R}^2$ endowed with the standard Euclidean metric $e$.
\end{conjecture}

Conjectures \ref{C1.1} and \ref{C1.2} are equivalent if multisheet domains are involved into the consideration. The correspondence between two kinds of Steklov spectra is established as follows. Choose a biholomorphism $\Phi:{\mathbb D}\rightarrow\Omega$. Then $\mbox{Sp}(\Omega)=\mbox{Sp}(a\Lambda_e)$, where $1/a=|\Phi'|_{\mathbb S}|$.

\bigskip

Given a Riemannian metric $g$ on the unit disc ${\mathbb D}$, let $\Delta_g$ be its Laplace -- Beltrami operator. The DN-operator of the metric is defined by
$$
\Lambda_g:C^\infty({\mathbb S})\rightarrow C^\infty({\mathbb S}),\quad \Lambda_g(f)=-\left.\partial u/\partial\nu\right|_{\mathbb S},
$$
where $\nu$ is the unit outward normal to ${\mathbb S}$ with respect to the metric $g$ and $u$ is the solution to the Dirichlet problem
$$
\Delta_gu=0\quad\mbox{in}\quad {\mathbb D},\quad u|_{\mathbb S}=f.
$$
This coincides with (\ref{1.1}) in the case of the Euclidean metric. Again, the eigenvalue spectrum $\mbox{Sp}(\Lambda_g)$ is discrete and non-negative. We again pose the inverse problem:  to what extent is a Riemannian metric $g$ on ${\mathbb D}$ determined by the spectrum $\mbox{Sp}(\Lambda_g)$? Here, the natural conjecture sounds as follows:

\begin{conjecture} \label {C1.3}
A Riemannian metric on the unit disk is determined by its Steklov spectrum uniquely up to a conformal equivalence. More precisely, for two Riemannian metrics $g$ and $g'$ on ${\mathbb D}$, the equality
$\mbox{\rm Sp}(\Lambda_g)=\mbox{\rm Sp}(\Lambda_{g'})$ holds if and only if there exist a diffeomorphism $\Psi:{\mathbb D}\rightarrow {\mathbb D}$ and function $0<\rho\in C^\infty({\mathbb D})$ such that $\rho|_{\mathbb S}=1$ and
$g'=\rho\Psi^*g$.
\end{conjecture}

Conjectures \ref{C1.1} and \ref{C1.3} are equivalent, as is proved in \cite{JS}. The first version of the inverse problem seems easier from the analytic viewpoint since it is a problem of recovering one function of one real argument. On the other hand, two last versions seem, probably, more interesting from the geometric viewpoint. Of course, any progress in one of these problems would imply the corresponding results for two other problems.

\section{Zeta-invariants}

Our main construction is actually a generalization of arguments by Edward \cite[Theorem 2]{E}. Recall that ${\mathbb S}=\{e^{i\theta}\}$ is the unit circle. For a function $a\in C^\infty(\mathbb S)$, let ${\hat a}_n$ be its Fourier coefficients, i.e.,
$$
a(\theta)=\sum\limits_{n=-\infty}^\infty {\hat a}_n e^{in\theta}.
$$
For every integer $k\geq 1$, we define
\begin{equation}
Z_k(a)=\sum\limits_{j_1+\dots +j_{2k}=0} N_{j_1\dots j_{2k}}\,{\hat a}_{j_1}{\hat a}_{j_2}\dots {\hat a}_{j_{2k}},
                                               \label{2.3}
\end{equation}
where, for $j_1+\dots+ j_{2k}=0$,
\begin{equation}
\begin{aligned}
N_{j_1\dots j_{2k}}=\sum\limits_{n=-\infty}^\infty
\Big[&\left|n(n+j_1)(n+j_1+j_2)\dots(n+j_1+\dots+ j_{2k-1})\right|\\
&-n(n+j_1)(n+j_1+j_2)\dots(n+j_1+\dots +j_{2k-1})\Big].
\end{aligned}
                                               \label{2.4}
\end{equation}
The quantities $Z_k(a)\ (k=1,2,\dots)$ will be called {\it zeta-invariants} of the function $a$ (or of the operator $a\Lambda_e$).
There is only a finite number of nonzero summands on the right-hand side of (\ref{2.4}) since the product
\begin{equation}
f(n)=n(n+j_1)(n+j_1+j_2)\dots(n+j_1+\dots+ j_{2k-1})
                                               \label{2.4.2}
\end{equation}
is a polynomial of degree $2k$ in $n$ which takes positive values for sufficiently large $|n|$.

Series (\ref{2.3}) absolutely converges since Fourier coefficients ${\hat a}_n$ fast decay. We will present corresponding estimates at the end of the current section.


We emphasize that definition (\ref{2.3}) makes sense for an arbitrary (complex-valued) function $a\in C^\infty({\mathbb S})$.
Thus, $Z_k(a)$ are explicitly expressed through Fourier coefficients of $a$, although in a rather complicated manner. On the other hand, in the case of a positive function $a$, zeta-invariants are determined by the eigenvalue spectrum of $a\Lambda_e$ as Theorem \ref{T2.1} below states. Before formulating the theorem, we need some preliminaries.

In the rest of the section, we consider a positive function $a\in C^\infty({\mathbb S})$ normalized by the condition
\begin{equation}
\frac{1}{2\pi}\int\limits_0^{2\pi}\frac{d\theta}{a(\theta)}=1.
                                               \label{2.0}
\end{equation}
Let $\{0=\lambda_0<\lambda_1\leq\lambda_2\leq\dots\}$ be the spectrum of the operator $a\Lambda_e$. {\it The zeta-function} of the operator is defined by
\begin{equation}
\zeta_a(s)=\mbox{Tr}\,[(a\Lambda_e)^{-s}]=\sum\limits_{n=1}^\infty\lambda_n^{-s}.
                                               \label{2.1}
\end{equation}
Recall \cite{E} that the asymptotics of spectra of $a\Lambda_e$ and of $\Lambda_e$ are the same. This implies that series (\ref{2.1}) converges for $\mbox{Re}\,s>1$ and $\zeta_a(s)$ extends to a meromorphic function on ${\mathbb C}$ with the unique simple pole at $s=1$. Moreover, $\zeta_a(s)-2\zeta_R(s)$ is an entire function, where
$$
\zeta_R(s)=\sum\limits_{n=1}^\infty n^{-s}
$$
is the classical Riemann zeta-function.

\begin{theorem} \label{T2.1}
For a  function $0<a\in C^\infty({\mathbb S})$ satisfying (\ref{2.0}) and for every $k\geq 1$,
$$
Z_k(a)=\zeta_a(-2k).
$$
\end{theorem}

To prove the theorem, we need the following

\begin{lemma} \label{L2.1}
Introduce the operator $D_\theta=-i\frac{d}{d\theta}:C^\infty({\mathbb S})\rightarrow C^\infty({\mathbb S})$ on the unit circle ${\mathbb S}=\{e^{i\theta}\}$.
For a function $0<a\in C^\infty({\mathbb S})$ satisfying (\ref{2.0}), the operator $aD_\theta$ is intertwined with $D_\theta$, i.e., there exists a diffeomorphism $\varphi:{\mathbb S}\rightarrow{\mathbb S}$ such that $aD_\theta=\varphi^*\circ D_\theta\circ\varphi^{*-1}$, where $\varphi^*u=u\circ\varphi$ for $u\in C^\infty({\mathbb S})$.
\end{lemma}

\begin{proof}
Define the diffeomorphism $\varphi:{\mathbb S}\rightarrow{\mathbb S}$ by
$$
\varphi(e^{i\theta})=\exp\left[i\int\limits_0^\theta\frac{d\tau}{a(\tau)}\right].
$$
Then $\frac{d\varphi}{d\theta}=a^{-1}(\theta)$.

For a function $u\in C^\infty({\mathbb S})$,
$$
(D_\theta\circ\varphi^*)u=D_\theta(u\circ\varphi)=(D_\theta u)\circ\varphi\cdot\frac{d\varphi}{d\theta}
=a^{-1}(D_\theta u)\circ\varphi=a^{-1}\varphi^*(D_\theta u)=a^{-1}(\varphi^*\circ D_\theta)(u).
$$
We have thus proved that
$$
a(D_\theta\circ\varphi^*)=\varphi^*\circ D_\theta.
$$
This can be rewritten in the form
$$
(aD_\theta)\circ\varphi^*=\varphi^*\circ D_\theta
$$
or
$$
aD_\theta=\varphi^*\circ D_\theta\circ\varphi^{*-1}.
$$
\end{proof}

By the lemma, the operators $(aD_\theta)^2$ and $D_\theta^2=\Lambda_e^2$ are intertwined and therefore
$$
\mbox{Tr}\,[(aD_\theta)^{2s}]=\mbox{Tr}\,[\Lambda_e^{2s}]\quad (\mbox{\rm Re}\,s<-1).
$$
In what follows, we will use just this relation.

\begin{proof}[Proof of Theorem \ref{T2.1}]
Recall that the classical Riemann zeta-function has zeros at even negative integers: $\zeta_R(-2k)=0\ (k=1,2,\dots)$. Therefore
$$
\zeta_a(-2k)=\zeta_a(-2k)-2\zeta_R(-2k)=\mbox{Tr}[(a\Lambda_e)^{2k}-D_\theta^{2k}].
$$
With the help of Lemma \ref{L2.1}, this implies
\begin{equation}
\zeta_a(-2k)=\mbox{Tr}[(a\Lambda_e)^{2k}-(aD_\theta)^{2k}].
                                               \label{2.5}
\end{equation}

We compute the right-hand side of (\ref{2.5}) by evaluating the operators $(a\Lambda_e)^{2k}$ and $(aD_\theta)^{2k}$  on the elements of the trigonometric basis $e^{in\theta}$.

By induction in $k$, we prove the formula
\begin{equation}
\begin{aligned}
(&a\Lambda_e)^{2k}e^{in\theta}=\sum\limits_{r_1,\dots,r_k}\ \sum\limits_{j_1+j_2=r_1-n}\ \sum\limits_{j_3+j_4=r_2-r_1}\dots\sum\limits_{j_{2k-1}+j_{2k}=r_k-r_{k-1}}\\
&\left|nr_1\dots r_{k-1}(n+j_1)(r_1+j_3)(r_2+j_5)\dots(r_{k-1}+j_{2k-1})\right|\,{\hat a}_{j_1}{\hat a}_{j_2}\dots {\hat a}_{j_{2k}}\,e^{ir_k\theta}.
\end{aligned}
                                               \label{2.6}
\end{equation}
We start with the obvious equality
$$
(a\Lambda_e)e^{in\theta}=|n|a\,e^{in\theta}.
$$
Applying the operator $a\Lambda_e$ to this equality, we obtain
\begin{equation}
\begin{aligned}
(a\Lambda_e)^{2}e^{in\theta}&=|n|a\Lambda_e(ae^{in\theta})=|n|a\Lambda_e\Big(\sum\limits_{j_1} {\hat a}_{j_1}e^{i(n+j_1)\theta}\Big)\\
&=|n|a\sum\limits_{j_1} {\hat a}_{j_1}|n+j_1|\,e^{i(n+j_1)\theta}=\sum\limits_{j_2} {\hat a}_{j_2} e^{ij_2\theta}\sum\limits_{j_1} {\hat a}_{j_1}|n(n+j_1)|\,e^{i(n+j_1)\theta}\\
&=\sum\limits_r\Big(\sum\limits_{j_1+j_2=r-n}|n(n+j_1)|\, {\hat a}_{j_1}{\hat a}_{j_2}\Big)e^{ir\theta}.
\end{aligned}
                                               \label{2.7}
\end{equation}
This coincides with (\ref{2.6}) for $k=1$.

Now, we are doing the induction step. Apply the operator $(a\Lambda_e)^{2}$ to formula (\ref{2.6})
$$
\begin{aligned}
(a\Lambda_e&)^{2(k+1)}e^{in\theta}=\sum\limits_{r_1,\dots,r_k}\ \sum\limits_{j_1+j_2=r_1-n}\ \sum\limits_{j_3+j_4=r_2-r_1}\dots\sum\limits_{j_{2k-1}+j_{2k}=r_k-r_{k-1}}\\
&\left|nr_1\dots r_{k-1}(n+j_1)(r_1+j_3)\dots(r_{k-1}+j_{2k-1})\right|\,{\hat a}_{j_1}\dots {\hat a}_{j_{2k}}\,(a\Lambda_e)^{2}e^{ir_k\theta}
\end{aligned}
$$
and use (\ref{2.7}) to obtain
$$
\begin{aligned}
(a\Lambda_e)^{2(k+1)}e^{in\theta}&=\sum\limits_{r_1,\dots,r_k}\ \sum\limits_{j_1+j_2=r_1-n}\ \sum\limits_{j_3+j_4=r_2-r_1}\dots\sum\limits_{j_{2k-1}+j_{2k}=r_k-r_{k-1}}\\
&\left|nr_1\dots r_{k-1}(n+j_1)(r_1+j_3)\dots(r_{k-1}+j_{2k-1})\right|\,{\hat a}_{j_1}\dots {\hat a}_{j_{2k}}\\
&\times\sum\limits_{r_{k+1}}\ \sum\limits_{j_{2k+1}+j_{2k+2}=r_{k+1}-r_k}|r_k(r_k+j_{2k+1})|\,{\hat a}_{j_{2k+1}}{\hat a}_{j_{2k+2}}e^{ir_{k+1}\theta}.
\end{aligned}
$$
After changing the order of summations, this gives (\ref{2.6}) for $k:=k+1$. Formula (\ref{2.6}) is thus proved.

The formula
\begin{equation}
\begin{aligned}
(&aD_\theta)^{2k}e^{in\theta}=\sum\limits_{r_1,\dots,r_k}\ \sum\limits_{j_1+j_2=r_1-n}\ \sum\limits_{j_3+j_4=r_2-r_1}\dots\sum\limits_{j_{2k-1}+j_{2k}=r_k-r_{k-1}}\\
&nr_1\dots r_{k-1}(n+j_1)(r_1+j_3)(r_2+j_5)\dots(r_{k-1}+j_{2k-1})\,{\hat a}_{j_1}\dots {\hat a}_{j_{2k}}\,e^{ir_k\theta}
\end{aligned}
                                               \label{2.8}
\end{equation}
is proved in the same way as (\ref{2.6}) has been proved. We actually do not need to repeat the proof. All we need is to compare the equalities
$$
(a\Lambda_e)e^{in\theta}=|n|ae^{in\theta},\quad (aD_\theta)e^{in\theta}=nae^{in\theta}.
$$
Therefore all formulas for $a\Lambda_e$ are valid for $aD_\theta$ with modulus signs omitted.

Taking the difference of (\ref{2.6}) and (\ref{2.8}), we obtain
\begin{equation}
\begin{aligned}
\big[(a\Lambda_e)^{2k}-(aD_\theta)^{2k}\big]e^{in\theta}&=\sum\limits_{r_1,\dots,r_k}\ \sum\limits_{j_1+j_2=r_1-n}\ \sum\limits_{j_3+j_4=r_2-r_1}\dots\sum\limits_{j_{2k-1}+j_{2k}=r_k-r_{k-1}}\\
&N(n;r_1,\dots,r_{k-1};j_1,j_3\dots,j_{2k-1})\,{\hat a}_{j_1}\dots {\hat a}_{j_{2k}}\,e^{ir_k\theta},
\end{aligned}
                                               \label{2.9}
\end{equation}
where the temporary notation
$$
\begin{aligned}
N(n;r_1,\dots,r_{k-1};j_1,j_3\dots,j_{2k-1})
=|&nr_1\dots r_{k-1}(n+j_1)(r_1+j_3)\dots(r_{k-1}+j_{2k-1})|\\ -&nr_1\dots r_{k-1}(n+j_1)(r_1+j_3)\dots(r_{k-1}+j_{2k-1})
\end{aligned}
$$
is used.

To evaluate the trace of $(a\Lambda_e)^{2k}-(aD_\theta)^{2k}$, we have to distinguish the coefficient at $e^{in\theta}$ on the right-hand side of (\ref{2.9}), i.e., to set $r_k=n$, and then to implement the summation over $n$
\begin{equation}
\begin{aligned}
\mbox{Tr}\,\big[(a\Lambda_e)^{2k}-(aD_\theta)^{2k}\big]&=\sum\limits_n\sum\limits_{r_1,\dots,r_{k-1}}\ \sum\limits_{j_1+j_2=r_1-n}\ \sum\limits_{j_3+j_4=r_2-r_1}\dots\sum\limits_{j_{2k-3}+j_{2k-2}=r_{k-1}-r_{k-2}}\\&\sum\limits_{j_{2k-1}+j_{2k}=n-r_{k-1}}
N(n;r_1,\dots,r_{k-1};j_1,j_3,\dots,j_{2k-1})\,{\hat a}_{j_1}\dots {\hat a}_{j_{2k}}.
\end{aligned}
                                               \label{2.10}
\end{equation}

Now, we change the order of summations in (\ref{2.10}) to move the summation over $n$ to the most inner position (one can easily justify the change of the summation order). To this end, for a fixed $n$, we set
$$
\begin{aligned}
&r_1=j_1+j_2+n=n+j_1+j_2,\\
&r_2=j_3+j_4+r_1=n+j_1+j_2+j_3+j_4,\\
&\dots\dots\dots\dots\\
&r_{k-1}=j_{2k-3}+j_{2k-2}+r_{k-2}=n+j_1+j_2+\dots+j_{2k-2}.
\end{aligned}
$$
Then (\ref{2.10}) takes the form
\begin{equation}
\mbox{Tr}\,\big[(a\Lambda_e)^{2k}-(aD_\theta)^{2k}\big]=\sum\limits_{j_1+\dots+j_{2k}=0}\ \sum\limits_n
\tilde N(n;j_1,\dots,j_{2k-1})\,{\hat a}_{j_1}\dots {\hat a}_{j_{2k}},
                                               \label{2.11}
\end{equation}
where
$$
\begin{aligned}
&\tilde N(n;j_1,\dots,j_{2k-1})\\
&=N(n;n+j_1+j_2,n+j_1+j_2+j_3+j_4,\dots,n+j_1+\dots+j_{2k-2};j_1,j_3,\dots,j_{2k-1})\\
&=\left|n(n+j_1)(n+j_1+j_2)\dots(n+j_1+j_2+\dots +j_{2k-1})\right|\\
&-n(n+j_1)(n+j_1+j_2)\dots(n+j_1+j_2+\dots+ j_{2k-1}).
\end{aligned}
$$

The right-hand side of (\ref{2.11}) coincides with the right-hand side of (\ref{2.3}). We have thus proved
$$
\mbox{Tr}\,\big[(a\Lambda_e)^{2k}-(aD_\theta)^{2k}\big]=Z_k(a).
$$
Together with (\ref{2.5}), this gives the statement of the theorem.
\end{proof}

Let us discuss series (\ref{2.3}) in more details.

Coefficients (\ref{2.4}) are even in the following sense:
\begin{equation}
N_{-j_1,\dots, -j_{2k}}=N_{j_1\dots j_{2k}}\quad (j_1+\dots+ j_{2k}=0).
                                               \label{2.4'}
\end{equation}
This is proved by the change $m=-n$ of the summation index in (\ref{2.4}). These coefficients are also invariant under the cyclic permutation of all indices:
\begin{equation}
N_{j_1 j_2\dots j_{2k}}=N_{j_2 j_3\dots j_{2k} j_1}=\dots=N_{j_{2k}j_1 \dots j_{2k-1}}\quad (j_1+\dots+ j_{2k}=0).
                                               \label{2.4.3}
\end{equation}
This is proved by the change $m=n+j_1$ of the summation index in (\ref{2.4}). But, in the general case, $N_{j_1\dots j_{2k}}$ are not invariant under an arbitrary permutation of indices.

It makes sense to symmetrize the $2k$-form (\ref{2.3}), i.e., to rewrite it in the form
\begin{equation}
Z_k(a)=\sum\limits_{j_1,\dots, j_{2k}=-\infty}^\infty Z_{j_1\dots j_{2k}}\,{\hat a}_{j_1}\dots {\hat a}_{j_{2k}},
                                               \label{2.3.1}
\end{equation}
where
$$
Z_{j_1\dots j_{2k}}=0\quad\mbox{for}\quad j_1+\dots +j_{2k}\neq 0
$$
and
\begin{equation}
Z_{j_1\dots j_{2k}}=\frac{1}{(2k)!}\sum\limits_{\pi\in\Pi_{2k}} N_{j_{\pi(1)}\dots j_{\pi(2k)}}\quad\mbox{for}\quad j_1+\dots +j_{2k}= 0.
                                               \label{2.3.3}
\end{equation}
Here $\Pi_{2k}$ is the group of all permutations of the set $\{1,2,\dots,2k\}$.

The coefficients $Z_{j_1\dots j_{2k}}$ are symmetric, i.e., invariant under an arbitrary permutation of the indices $(j_1,\dots, j_{2k})$.
Of course, the symmetrization preserves property (\ref{2.4'}), i.e.,
\begin{equation}
Z_{-j_1,\dots, -j_{2k}}= Z_{j_1\dots j_{2k}}.
                                               \label{2.4''}
\end{equation}
This implies the important statement: all zeta-invariants are real in the case of a real function $a$. Indeed, applying the complex conjugation to (\ref{2.3.1}) and taking the reality of $Z_{j_1\dots j_{2k}}$ into account, we obtain
$$
\overline{Z_k(a)}=Z_k(\bar a)=\sum\limits_{j_1,\dots, j_{2k}=-\infty}^\infty Z_{j_1\dots j_{2k}}\,\overline{{\hat a}_{j_1}}\dots \overline{{\hat a}}_{j_{2k}}.
$$
Fourier coefficients of a real function satisfy $\overline{{\hat a}_{j}}={\hat a}_{-j}$ and the last formula becomes
$$
\overline{Z_k(a)}=\sum\limits_{j_1,\dots ,j_{2k}=-\infty}^\infty Z_{-j_1,\dots,- j_{2k}}\, {\hat a}_{j_1}\dots {\hat a}_{j_{2k}}.
$$
The right-hand side of this formula coincides with the right-hand side of (\ref{2.3.1}) since the coefficients are even.

Formula (\ref{2.3.3}) can be simplified a little bit with the help of (\ref{2.4.3}). Indeed, let $\Pi_{2k-1}$ be the subgroup of $\Pi_{2k}$ consisting of all permutations fixing the last element, i.e.,
$$
\Pi_{2k-1}=\{\pi=(\pi(1),\dots,\pi(2k-1),2k)\}\subset\Pi_{2k}.
$$
Let $\zeta=(2,3,\dots,2k,1)$ be the cyclic permutation. Represent $\Pi_{2k}$ as the union of residue classes
$$
\Pi_{2k}=\bigcup\limits_{\ell=0}^{2k-1}\zeta^\ell\Pi_{2k-1}
$$
and separate summands of (\ref{2.3.3}) to $2k$ groups according to the representation. By (\ref{2.4.3}), these partial sums coincide and formula (\ref{2.3.3}) simplifies to the following one:
\begin{equation}
Z_{j_1\dots j_{2k}}=\frac{1}{(2k-1)!}\sum\limits_{\pi\in\Pi_{2k-1}} N_{j_{\pi(1)}\dots j_{\pi(2k-1)}j_{2k}}\quad\mbox{for}\quad j_1+\dots +j_{2k}= 0.
                                               \label{2.3.4}
\end{equation}

Let us prove the convergence of series (\ref{2.3.1}). To this end we will first derive the following estimate for the coefficients of the series:
\begin{equation}
0\leq Z_{j_1\dots j_{2k}}\leq 2\big(2(|j_1|+\dots+|j_{2k}|)\big)^{2k+1}.
                                               \label{2.4.1}
\end{equation}
Indeed, let us fix $(j_1,\dots,j_{2k})$ and set $|j|=|j_1|+\dots+|j_{2k}|$. Let $x_-$ and $x_+$ be the minimal and maximal roots of the polynomial  $f(n)$ defined by (\ref{2.4.2}). The roots satisfy $|x_\pm|\leq |j|$. A summand of (\ref{2.4}) can be non-zero only if $n\in(x_-,x_+)$, the number of such summands is $\leq 2|j|$. The value of each summand is not more than
$
2(|n|+|j|)^{2k}\leq 2(2|j|)^{2k}.
$
Therefore
$$
N_{j_1\dots j_{2k}}\leq 2(2|j|)^{2k}(2|j|)=2(2|j|)^{2k+1}.
$$
This proves (\ref{2.4.1}).

Fourier coefficients of a smooth function $a$ fast decay, i.e., satisfy
$$
|{\hat a}_n|\leq C_M(|n|+1)^{-M}
$$
for arbitrary $M>0$. Together with (\ref{2.4.1}), this implies the absolute convergence of series (\ref{2.3.1}). Indeed,
$$
|Z_{j_1\dots j_{2k}}\,{\hat a}_{j_1}\dots {\hat a}_{j_{2k}}|\leq 2^{2k+2}C_M^{2k}(|j|+1)^{-M+2k+1},\quad\mbox{where}\quad |j|=|j_1|+\dots+|j_{2k}|.
$$
Therefore
$$
\sum\limits_{j_1+\dots +j_{2k}=0}|Z_{j_1\dots j_{2k}}\,{\hat a}_{j_1}\dots {\hat a}_{j_{2k}}|\leq 2^{2k+2}C_M^{2k}\sum\limits_{\ell=0}^\infty (\ell+1)^{-M+2k+1}K(\ell),
$$
where
$$
K(\ell)=\sharp\{(j_1,\dots,j_{2k})\mid j_1+\dots +j_{2k}=0,\ |j_1|+\dots+|j_{2k}|=\ell\}\leq(2\ell+1)^{2k}.
$$
Finally,
$$
\sum\limits_{j_1+\dots +j_{2k}=0}|Z_{j_1\dots j_{2k}}\,{\hat a}_{j_1}\dots {\hat a}_{j_{2k}}|\leq 2^{4k+2}C_M^{2k}\sum\limits_{\ell=0}^\infty (\ell+1)^{-M+4k+1}.
$$
The series on the right-hand side converges if $M$ is sufficiently large.

The first zeta-invariant was actually introduced by Edward \cite{E}. Let us reproduce his calculations here. By definition (\ref{2.3})--(\ref{2.4}),
\begin{equation}
Z_1(a)=\sum\limits_{j+\ell=0}Z_{j\ell}\,a_ja_\ell=\sum\limits_{j}Z_{j,-j}\,{\hat a}_j{\hat a}_{-j},
                                               \label{2.12}
\end{equation}
where
$$
Z_{j,-j}=N_{j,-j}=\sum\limits_n\Big(|n(n+j)|-n(n+j)\Big).
$$
Obviously,
$$
|n(n+j)|-n(n+j)=\left\{\begin{array}{ll}
-2n(n+j),\quad&\mbox{if}\quad 0<n<-j,\\
-2n(n+j),\quad&\mbox{if}\quad -j<n<0,\\
0&\mbox{otherwise}.
\end{array}\right.
$$
Therefore, for a positive $j$,
$$
Z_{j,-j}=-2\sum\limits_{n=-j}^{-1}n(n+j)=-2\sum\limits_{n=-j}^{-1}n^2-2j\sum\limits_{n=-j}^{-1}n
=-2\sum\limits_{n=1}^{j}n^2+2j\sum\limits_{n=1}^{j}n=\frac{1}{3}(j^3-j).
$$
Here, we have used the equalities
\begin{equation}
\sum\limits_{n=1}^{j}n=\frac{1}{2}j(j+1),\quad\sum\limits_{n=1}^{j}n^2=\frac{1}{6}j(j+1)(2j+1).
                                               \label{2.15}
\end{equation}
Similarly, $Z_{j,-j}=\frac{1}{3}|j^3-j|$ for a negative $j$. Thus, for all $j$,
\begin{equation}
Z_{j,-j}=\frac{1}{3}|j^3-j|
                                               \label{2.12'}
\end{equation}
Substituting these values into (\ref{2.12}), we obtain
\begin{equation}
Z_1(a)=\frac{1}{3}\sum\limits_{j=-\infty}^\infty|j^3-j|\, {\hat a}_j{\hat a}_{-j}=\frac{2}{3}\sum\limits_{n=2}^\infty(n^3-n)\, {\hat a}_n{\hat a}_{-n}.
                                               \label{2.13}
\end{equation}

\section{Conformal equivalence in terms of Fourier coefficients}

For $\rho\in(-1,1)$, let $\Phi_\rho$ be the conformal transformation of the unit disk defined by
\begin{equation}
\Phi_\rho(z)=\frac{z-\rho}{1-\rho z}.
                                               \label{3.2}
\end{equation}
Given a function $a\in C^\infty({\mathbb S})$, let the function $b$ be conformally equivalent to $a$ via the conformal map $\Phi_\rho$, i.e.,
\begin{equation}
b=a\circ\varphi\left(\frac{d\varphi}{d\theta}\right)^{-1},\quad\mbox{where}\quad \varphi=\Phi_\rho|_{\mathbb S}.
                                               \label{4.0}
\end{equation}
This fact will be denoted by $b=a\Phi_\rho$, the notation will be explained in the next section.
As is seen from (\ref{4.0}), Fourier coefficients ${\hat b}_n$ of $b$ should depend linearly on Fourier coefficients ${\hat a}_n$ of $a$, i.e.,
$$
{\hat b}_n=\sum\limits_k\mu_{nk}(\rho){\hat a}_k.
$$
In the current section, we will evaluate the (infinite) matrix $M(\rho)=\big(\mu_{nk}(\rho)\big)_{n,k=-\infty}^\infty$ and establish some properties of the matrix.

By the definition of Fourier coefficients,
$$
{\hat b}_n=\frac{1}{2\pi}\int\limits_0^{2\pi}e^{-in\theta}b(\theta)\,d\theta=
\frac{1}{2\pi}\int\limits_0^{2\pi}e^{-in\theta}a(\varphi(\theta))\left(\frac{d\varphi}{d\theta}\right)^{-1}\,d\theta
$$
or
$$
{\hat b}_n=
\frac{1}{2\pi}\int\limits_0^{2\pi}e^{-in\theta(\varphi)}a(\varphi)\left(\frac{d\theta}{d\varphi}(\varphi)\right)^{2}\,d\varphi.
$$
Change the integration variable as
$$
z=e^{i\varphi},\quad d\varphi=\frac{1}{i}z^{-1}\,dz,\quad e^{i\theta(\varphi)}=\frac{z+\rho}{1+\rho z},\quad\frac{d\theta}{d\varphi}=\frac{1-\rho^2}{|1+\rho z|^2}.
$$
Then
$$
{\hat b}_n=(1-\rho^2)^2
\frac{1}{2\pi i}\oint\limits_{|z|=1}\left(\frac{1+\rho z}{z+\rho}\right)^nz^{-1}|1+\rho z|^{-4}a(z)\,dz.
$$
Substituting $a(z)=\sum_k{\hat a}_kz^k$, we arrive to the formula
\begin{equation}
{\hat b}_n=\sum\limits_{k=-\infty}^\infty\mu_{nk}{\hat a}_k,
                                               \label{4.1}
\end{equation}
where
\begin{equation}
\mu_{nk}=(1-\rho^2)^2
\frac{1}{2\pi i}\oint\limits_{|z|=1}\left(\frac{1+\rho z}{z+\rho}\right)^nz^{k-1}|1+\rho z|^{-4} \,dz.
                                               \label{4.2}
\end{equation}

We expand the last factor of the integrand of (\ref{4.2}) in powers of $z$ taking the relation $|z|=1$ into account
$$
\begin{aligned}
|1+\rho z|^2&=(1+\rho z)(1+\rho z^{-1})=\frac{1}{z}(1+\rho z)(z+\rho),\\
|1+\rho z|^{-4}&=z^2(1+\rho z)^{-2}(z+\rho)^{-2}.
\end{aligned}
$$
Substituting this expression into (\ref{4.2}), we obtain the final formula
\begin{equation}
\mu_{nk}=\mu_{nk}(\rho)=(1-\rho^2)^2
\frac{1}{2\pi i}\oint\limits_{|z|=1}\frac{(1+\rho z)^{n-2}}{(z+\rho)^{n+2}}z^{k+1}\,dz.
                                               \label{4.3}
\end{equation}

We introduce also the constant matrix $D=(d_{nk})_{n,k=-\infty}^\infty$,
\begin{equation}
d_{nk}=(n-2)\delta_{n-1,k}-(n+2)\delta_{n+1,k}=\left\{
\begin{array}{cll}
n-2\quad&\mbox{if}\quad&k=n-1,\\
-(n+2)\quad&\mbox{if}\quad&k=n+1,\\
0&&\mbox{otherwise}.
\end{array}\right.
                                               \label{4.40}
\end{equation}

\begin{proposition} \label{P3.1}
The matrix $M(\rho)$ is expressed through $D$ and $\rho\in(-1,1)$ by
\begin{equation}
M(\rho)=e^{tD},\quad\mbox{where}\quad \tanh t=\rho.
                                               \label{4.41}
\end{equation}
The map $\rho\mapsto M(\rho)$ satisfies
\begin{equation}
M(\rho)M(\rho')=M(\rho''),\quad\mbox{where}\quad \rho''=\frac{\rho+\rho'}{1+\rho\rho'}.
                                               \label{4.42}
\end{equation}
In particular, the matrices $M(\rho)$ and $M(\rho')$ commute as well as $M(\rho)$ commutes with $D$.
\end{proposition}

\begin{proof}
It suffices to prove (\ref{4.41}), other statements of the proposition follow from (\ref{4.41}). Formula (\ref{4.41}) is equivalent to the differential equation
$$
\frac{dM}{dt}=DM.
$$
On assuming the variables $\rho$ and $t$ to be related by $\tanh t=\rho$, we rewrite the latter equation in the equivalent form
$$
(1-\rho^2)\frac{dM}{d\rho}=DM.
$$
Substituting value (\ref{4.40}), we write this in the form
\begin{equation}
(1-\rho^2)\frac{d\mu_{nk}}{d\rho}-(n-2)\mu_{n-1,k}+(n+2)\mu_{n+1,k}=0.
                                               \label{4.43}
\end{equation}
Thus, all we need is to prove the validity of (\ref{4.43}).

Differentiate (\ref{4.3}) to obtain
$$
\begin{aligned}
\frac{d\mu_{nk}}{d\rho}=\frac{(1-\rho^2)}{2\pi i}\oint\limits_{|z|=1}&\frac{(1+\rho z)^{n-3}}{(z+\rho)^{n+3}}\Big[-4\rho(1+\rho z)(z+\rho)\\
&+(n-2)(1-\rho^2)z(z+\rho)-(n+2)(1-\rho^2)(1+\rho z)\Big] z^{k+1}\,dz.
\end{aligned}
$$
Substituting this value of $d\mu_{nk}/d\rho$ and value (\ref{4.3}) of $\mu_{n\pm 1,k}$, we evaluate
$$
(1-\rho^2)\frac{d\mu_{nk}}{d\rho}-(n-2)\mu_{n-1,k}+(n+2)\mu_{n+1,k})=\frac{(1-\rho^2)^2}{2\pi i}\oint\limits_{|z|=1}\frac{(1+\rho z)^{n-3}}{(z+\rho)^{n+3}}f(n,\rho,z)
 z^{k+1}\,dz,
$$
where
$$
\begin{aligned}
f(n,\rho,z)=
-&4\rho(1+\rho z)(z+\rho)
+(n-2)(1-\rho^2)z(z+\rho)-(n+2)(1-\rho^2)(1+\rho z)\\-&(n-2)(z+\rho)^2+(n+2)(1+\rho z)^2.
\end{aligned}
$$
As one can easily check, $f(n,\rho,z)$ is identically equal to zero.
\end{proof}

\bigskip

Integral (\ref{4.3}) can be evaluated with the help of the residue theorem. First of all,
in the case of $n\leq-2$ and $k\geq-1$, the integrand of (\ref{4.3}) is a holomorphic function in the unit disk and therefore
\begin{equation}
\mu_{nk}=0\quad\mbox{for}\quad n\leq-2\quad\mbox{and}\quad k\geq-1.
                                               \label{4.4}
\end{equation}

Change the integration variable in (\ref{4.3}) as $z=1/\zeta$

$$
\begin{aligned}
\mu_{nk}&=(1-\rho^2)^2\frac{1}{2\pi i}\oint\limits_{|\zeta|=1}\frac{(1+\rho/\zeta)^{n-2}}{(1/\zeta+\rho)^{n+2}}\zeta^{-k-3}\,d\zeta\\
&=(1-\rho^2)^2\frac{1}{2\pi i}\oint\limits_{|\zeta|=1}\frac{(\zeta+\rho)^{n-2}}{(1+\rho\,\zeta)^{n+2}}\zeta^{-k+1}\,d\zeta\\
&=(1-\rho^2)^2\frac{1}{2\pi i}\oint\limits_{|\zeta|=1}\frac{(1+\rho\,\zeta)^{-n-2}}{(\zeta+\rho)^{-n+2}}\zeta^{-k+1}\,d\zeta=\mu_{-n,-k}.
\end{aligned}
$$
We have thus proved that
\begin{equation}
\mu_{nk}=\mu_{-n,-k}.
                                               \label{4.5}
\end{equation}
Together with (\ref{4.4}) this gives
\begin{equation}
\mu_{nk}=0\quad\mbox{for}\quad n\geq 2\quad\mbox{and}\quad k\leq 1.
                                               \label{4.6}
\end{equation}
Because of (\ref{4.5}), it suffices to consider the case of $n\geq 0$ only.

Assuming  $k\geq-1$, as is seen from (\ref{4.3}),
\begin{equation}
\mu_{nk}=(1-\rho^2)^2\mbox{Res}\left[\frac{(1+\rho z)^{n-2}}{(z+\rho)^{n+2}}z^{k+1}\right]_{z=-\rho} \quad(k\geq -1).
                                               \label{4.7}
\end{equation}
To find the residue, we have to expand the function
$$
\frac{(1+\rho z)^{n-2}}{(z+\rho)^{n+2}}z^{k+1}
$$
in powers of $(z+\rho)$.

First of all,
\begin{equation}
z^{k+1}=\sum\limits_{\ell=0}^{k+1}(-1)^{k-\ell+1}{{k+1}\choose\ell}\rho^{k-\ell+1}(z+\rho)^\ell.
                                               \label{3.19}
\end{equation}
Hereafter ${r\choose s}=\frac{r!}{s!(r-s)!}$ is the binomial coefficient that is assumed to be defined for all integers $r$ and $s$ under the agreement
\begin{equation}
{r\choose s}=0\quad\mbox{if}\quad r<0\quad\mbox{or}\quad s<0\quad\mbox{or}\quad s>r.
                                               \label{3.13}
\end{equation}

Next, from the equality $1+\rho z=\rho(z+\rho)+(1-\rho^2)$, we obtain
\begin{equation}
(1+\rho z)^{n-2}=\sum\limits_{p=0}^{n-2}{{n-2}\choose p}\rho^p(1-\rho^2)^{n-p-2}(z+\rho)^p.
                                               \label{3.19'}
\end{equation}

By (\ref{3.19}) and (\ref{3.19'}), assuming $n\geq 2$,
$$
\begin{aligned}
&\frac{(1+\rho z)^{n-2}z^{k+1}}{(z+\rho)^{n+2}}=\\=&\sum\limits_{m=0}^{n+k-1}\left(
\sum\limits_{\ell+p=m}(-1)^{k-\ell+1}{{n-2}\choose p}{{k+1}\choose\ell}\rho^{k+p-\ell+1}(1-\rho^2)^{n-p-2}\right)(z+\rho)^{m-n-2}.
\end{aligned}
$$
According to (\ref{4.7}), we have to take the coefficient at $(z+\rho)^{-1}$ on the right-hand side of this formula, i.e., to set $m=n+1$. We thus obtain
$$
\mu_{nk}=(1-\rho^2)^2
\sum\limits_{\ell+p=n+1}(-1)^{k-\ell+1}{{n-2}\choose p}{{k+1}\choose\ell}\rho^{k+p-\ell+1}(1-\rho^2)^{n-p-2}.
$$
Setting $p=n-\ell+1$, we obtain the final formula
\begin{equation}
\mu_{nk}=(-1)^{k+1}\frac{\rho^{n+k+2}}{1-\rho^2}\sum\limits_{\ell}(-1)^{\ell}{{n-2}\choose {\ell-3}}{{k+1}\choose\ell}\rho^{-2\ell}(1-\rho^2)^\ell\quad(n\geq 2,\ k\geq -1).
                                               \label{4.8}
\end{equation}
The summation in (\ref{4.8}) is actually performed in the limits
\begin{equation}
3\leq \ell\leq\mbox{min}\,(n+1,k+1).
                                               \label{4.8'}
\end{equation}

Formulas (\ref{4.6}) and (\ref{4.8}) explicitly express $\mu_{nk}$ for $n\geq 2$ and all $k$. Together with (\ref{4.5}) this gives $\mu_{nk}$ for $|n|\geq 2$ and all $k$. It remains to consider the cases of $n=0,\pm1$. We just present the results for the latter cases which are obtained by the same calculations as above.

\begin{equation}
\mu_{-1,k}=\frac{(-\rho)^{k+1}}{1-\rho^2}\quad\mbox{for}\quad k\geq -1;
                                               \label{4.16}
\end{equation}
\begin{equation}
\mu_{0,k}=\mu_{0,-k}=\frac{(-\rho)^k}{1-\rho^2}\big((k+1)-(k-1)\rho^2\big)\quad\mbox{for}\quad k\geq -1;
                                               \label{4.17}
\end{equation}
\begin{equation}
\mu_{1,k}=\frac{(-\rho)^{k-1}}{1-\rho^2}\Big(\frac{k(k+1)}{2}-(k^2-1)\rho^2+\frac{k(k-1)}{2}\rho^4\Big)\quad\mbox{for}\quad k\geq -1.
                                               \label{4.18}
\end{equation}
In particular,
\begin{equation}
\mu_{n,-1}=\mu_{n,0}=\mu_{n,1}=0\quad\mbox{for}\quad |n|\geq 2
                                               \label{4.15}
\end{equation}
and
\begin{equation}
\left(\begin{array}{ccc}
\mu_{-1,-1}&\mu_{-1,0}&\mu_{-1,1}\\
\mu_{0,-1}&\mu_{0,0}&\mu_{0,1}\\
\mu_{1,-1}&\mu_{1,0}&\mu_{1,1}
\end{array}\right)
=\frac{1}{1-\rho^2}
\left(\begin{array}{ccc}
1&-\rho&\rho^2\\
-2\rho&1+\rho^2&-2\rho\\
\rho^2&-\rho&1
\end{array}\right).
                                               \label{4.10}
\end{equation}
As one easily sees, $(1/2,1,1/2)^t$ is an eigenvector of matrix (\ref{4.10}) associated with the eigenvalue $\frac{1-\rho}{1+\rho}$. This means that, for every $\rho\in(0,1)$, the function
$1+\cos\theta$ is an eigenfunction of the operator $a\mapsto a\Phi_\rho$ associated with the eigenvalue $\frac{1-\rho}{1+\rho}$.

In the next section, we will need the estimate
\begin{equation}
|\mu_{n,k}(\rho)|\leq C_k|n|^{|k|}|\rho|^{|n|/2}\quad\mbox{for}\quad |n|\geq 2|k|\geq 2.
                                               \label{4.50}
\end{equation}
It easily follows from (\ref{4.8}). Indeed, let us first assume that $n\geq 2$ and $k\geq 0$. We derive from (\ref{4.8})
$$
|\mu_{nk}|\leq\sum\limits_{\ell}{{n-2}\choose {\ell-3}}{{k+1}\choose\ell}|\rho|^{n+2k-2\ell+2}(1-\rho^2)^{\ell-1}.
$$
As we have mentioned, the summation is actually performed over $\ell$ satisfying (\ref{4.8'}). Therefore the last factor on the right-hand side is bounded by 1 from above and the estimate  simplifies to the following one:
$$
|\mu_{nk}|\leq\sum\limits_{\ell}{{n-2}\choose {\ell-3}}{{k+1}\choose\ell}|\rho|^{n+2k-2\ell+2}.
$$
Assuming $n\geq 2k\geq 0$, (\ref{4.8'}) implies $n+2k-2\ell+2\geq n/2$ and our estimate takes the form
$$
|\mu_{nk}|\leq|\rho|^{n/2}\sum\limits_\ell {{n-2}\choose{\ell-3}}{{k+1}\choose{\ell}}.
$$
Finally,
$$
{{n-2}\choose{\ell-3}}=\frac{(n-2)(n-3)\dots(n-\ell+2)}{(\ell-3)!}\leq\frac{n^{\ell-2}}{(\ell-3)!}\leq\frac{n^{k-1}}{(\ell-3)!}
$$
and
$$
|\mu_{nk}|\leq  n^{k-1}|\rho|^{n/2}\sum\limits_{\ell=3}^{k+1} \frac{1}{(\ell-3)!}{{k+1}\choose{\ell}}=C_k n^{k-1} |\rho|^{n/2}.
$$
This proves (\ref{4.50}) in the case of $n\geq 2$ and $k\geq 0$. In the case of $n\geq 2$ and $k\leq 0$, estimate (\ref{4.50}) is trivially valid in virtue of (\ref{4.6}). Finally, to prove (\ref{4.50}) in the case of a negative $n$, it suffices to remember the evenness property (\ref{4.5}).

\section{Zeta-invariants and the conformal group}

If the operators $a\Lambda_e$ and $b\Lambda_e$ are isospectral for two positive functions $a,b\in C^\infty({\mathbb S})$, then by Theorem \ref{T2.1},
\begin{equation}
Z_k(a)=Z_k(b)\quad(k=1,2,\dots).
                                               \label{3.1}
\end{equation}
In particular, (\ref{3.1}) holds for conformally equivalent positive functions $a$ and $b$.

Let $G$ be the group of all conformal and anticonformal transformations of the unit disk ${\mathbb D}$ (it is a Lee group with two connected components, the component of unity is isomorphic to $PSL(2,{\mathbb R})$).
Restricting each transformation $\Phi\in G$ to ${\mathbb S}=\partial{\mathbb D}$, we consider $G$ as the three-dimensional Lee group of diffeomorphisms of the unit circle ${\mathbb S}$. Therefore the group $G$ acts from the right on the vector space $C^\infty({\mathbb S})$ by
\begin{equation}
a\Phi=a\circ\varphi\left|d\varphi/d\theta\right|^{-1}\quad\mbox{for}\quad \Phi\in G,\ a\in C^\infty(\gamma),\quad\mbox{where}\quad\varphi=\Phi|_{{\mathbb S}}.
                                               \label{5.-3}
\end{equation}
In these notations, formula (\ref{3.1}) means that
\begin{equation}
Z_k(a\Phi)=Z_k(a)\quad (\Phi\in G,\ k=1,2,\dots)
                                               \label{5.-2}
\end{equation}
for every positive function $a\in C^\infty({\mathbb S})$.

\begin{proposition} \label{P3.2}
Equality (\ref{5.-2}) holds for every function $a\in C^\infty({\mathbb S})$.
\end{proposition}

\begin{proof}
For fixed $k$ and $\Phi\in G$, set $Q(a)=Z_k(a\Phi)-Z_k(a)$. By definition (\ref{2.3})--(\ref{2.4}), $Q$ is a $2k$-form on $C^\infty({\mathbb S})$. We have to prove that the form is identically equal to zero. We consider $C^\infty({\mathbb S})$ as a topological vector space with the $C^\infty$-topology. The form $Z_k$ is continuous as our estimates at the end of Section 2 show. The form $Q$ is also continuous. We know that $Q(a)=0$ for a positive function $a$. Positive functions constitute an open convex cone in the space $C^\infty_{\mathbb R}({\mathbb S})$ of real functions. If a continuous form vanishes on an open set, then it is identically equal to zero. Thus,  $Q(a)=0$ for a every real function $a$.

Obviously, any $2k$-form on $C^\infty({\mathbb S})$ is uniquely determined by its restriction to $C^\infty_{\mathbb R}({\mathbb S})$. Therefore $Q$ is identically equal to zero.
\end{proof}

We are going to demonstrate that the conformal invariance (\ref{5.-2}) is equivalent to some linear relations between the coefficients $Z_{j_1\dots j_{2k}}$ of form (\ref{2.3.1}).

The group $G$ is generated by three subgroups:

(1) The group of rotations $R_\alpha:z\mapsto e^{i\alpha}z$.

(2) The group with two elements $\{I,J\}$, where $I$ is the identity and $J:z\mapsto\bar z$ is the complex conjugation.

(3) The group $T=\{\Phi_\rho\mid -1<\rho<1\}$, where $\Phi_\rho$ is defined by formula (\ref{3.2}).
From the viewpoint of hyperbolic geometry, $\Phi_\rho$ is the translation of the hyperbolic plane $\Big(\mbox{Int}\,D,\ ds^2=\frac{|dz|^2}{(1-|z|^2)^2}\Big)$ along the real line $(-1,1)$ on the distance $t$ such that $\rho=\tanh t$. This means that $\Phi_\rho(x)\in(-1,1)$ for $x\in(-1,1)$ and $\mbox{dist}\,(x,\Phi_\rho(x))=t$, where $\mbox{dist}$ means the hyperbolic distance. The translation $\Phi_\rho$ has two fixed points $\pm 1$ at the infinite line  ${\mathbb S}$.

There is no problem with the first two subgroups: the invariant $Z_k(a)$ does not change if the function $a$ is transformed either by a rotation or by the conjugation. Indeed, in such the case the factor $|d\varphi/d\theta|$ on the right-hand side of (\ref{5.-3}) is identically equal to 1. Therefore (\ref{5.-2}) is equivalent in this case to
\begin{equation}
Z_k(a\circ R_\alpha)=Z_k(a),\quad Z_k(a\circ J)=Z_k(a).
                                               \label{3.3}
\end{equation}
The Fourier coefficients of $a\circ R_\alpha$ are expressed through Fourier coefficients of $a$ by the formula
$$
(\widehat{a\circ R_\alpha})_j=e^{i\alpha j}{\hat a}_j.
$$
From this
$$
(\widehat{a\circ R_\alpha})_{j_1}\dots(\widehat{a\circ R_\alpha})_{j_{2k}}=e^{i\alpha(j_1+\dots+j_{2k})}\,{\hat a}_{j_1}\dots {\hat a}_{j_{2k}}={\hat a}_{j_1}\dots {\hat a}_{j_{2k}}\quad\mbox{if}\quad j_1+\dots+j_{2k}=0
$$
and all summands in (\ref{2.3}) do not change when $a$ is replaced by $a\circ R_\alpha$. Similarly,
$$
(\widehat{a\circ J})_j={\hat a}_{-j}
$$
and the second of equalities (\ref{3.3}) is equivalent to (\ref{2.4''}).

Thus, it remains to consider a translation $\Phi_\rho$ defined by (\ref{3.2}).
Let $a\in C^\infty({\mathbb S})$ and $b=a\Phi_\rho$. By (\ref{4.1}), Fourier coefficients of the functions $a$ and $b$ are related by
$$
{\hat b}_n=\sum\limits_{k=-\infty}^\infty\mu_{nk}(\rho){\hat a}_k.
$$
Substituting this expression into the formula
$$
Z_k(b)=\sum\limits_{j_1,\dots, j_{2k}=-\infty}^\infty Z_{j_1\dots j_{2k}}\,{\hat b}_{j_1}\dots {\hat b}_{j_{2k}},
$$
we obtain
$$
Z_k(b)=\sum\limits_{\ell_1,\dots, \ell_{2k}=-\infty}^\infty\Big(\sum\limits_{j_1,\dots, j_{2k}=-\infty}^\infty Z_{j_1\dots j_{2k}}\,
\mu_{j_1\ell_1}(\rho)\dots\mu_{j_{2k}\ell_{2k}}(\rho)\Big){\hat a}_{\ell_1}\dots {\hat a}_{\ell_{2k}}.
$$
Since $a$ is an arbitrary function, the equality $Z_k(a)=Z_k(b)$ is equivalent to the statement:
\begin{equation}
\sum\limits_{j_1,\dots, j_{2k}=-\infty}^\infty Z_{j_1\dots j_{2k}}\,\mu_{j_1\ell_1}(\rho)\dots\mu_{j_{2k}\ell_{2k}}(\rho)=Z_{\ell_1\dots \ell_{2k}}.
                                               \label{5.0}
\end{equation}
This equality should hold for every $k=1,2,\dots$, for every integer indices $(\ell_1,\dots, \ell_{2k})$, and for every $\rho\in (-1,1)$.

Let us demonstrate that the series on the left-hand side of (\ref{5.0}) absolutely converges.
We fix the indices $(\ell_1,\dots, \ell_{2k})$, set $\ell=|\ell_1|+\dots +|\ell_{2k}|$, and estimate from above the absolute value of the left-hand side of (\ref{5.0}) by
$$
\sum\limits_{j_1,\dots,j_{2k}=\ell-1}^{\ell+1} Z_{j_1\dots j_{2k}}\,\left|\mu_{j_1\ell_1}(\rho)\dots\mu_{j_{2k}\ell_{2k}}(\rho)\right|
+\sum\limits_{j=\ell+2}^\infty\sum\limits_{|j_1|+\dots +|j_{2k}|=j} Z_{j_1\dots j_{2k}}\,\left|\mu_{j_1\ell_1}(\rho)\dots\mu_{j_{2k}\ell_{2k}}(\rho)\right|.
$$
The first sum is finite. So, we have to check the convergence of the second series. To this end we use  (\ref{2.4.1}) and (\ref{4.50}) to obtain for $|j_1|+\dots+ |j_{2k}|=j,\ |j_\alpha|\geq\ell+2\ (1\leq\alpha\leq2k)$
$$
\begin{aligned}
\left|Z_{j_1\dots j_{2k}}\,\mu_{j_1\ell_1}(\rho)\dots\mu_{j_{2k}\ell_{2k}}(\rho)\right|&\leq C_kj^{2k+1}C_{\ell_1}|j_1|^{|\ell_1|}|\rho|^{|j_1|/2}\dots
C_{\ell_{2k}}|j_{2k}|^{|\ell_{2k}|}|\rho|^{|j_{2k}|/2}\\ &\leq C_{k,\ell_1\dots\ell_{2k}}|\rho|^{j/2+2k+1}.
\end{aligned}
$$
From this
$$
\sum\limits_{j=\ell+2}^\infty\sum\limits_{|j_1|+\dots +|j_{2k}|=j} Z_{j_1\dots j_{2k}}\,\left|\mu_{j_1\ell_1}(\rho)\dots\mu_{j_{2k}\ell_{2k}}(\rho)\right|\leq
C_{k,\ell_1\dots\ell_{2k}}\sum\limits_{j=\ell+2}^\infty(j+1)^{4k+1}|\rho|^{j/2+2k+1}.
$$
The series on the right-hand side converges for $|\rho|<1$.

Equation (\ref{5.0}) obviously holds for $\rho=0$ since $M(0)=I$. We differentiate (\ref{5.0}) with respect to $\rho$. The differentiation can be justified with the help of the same estimates as have been used in the previous paragraph. In this way we obtain the following equation equivalent to (\ref{5.0}):
$$
\sum\limits_{j_1,\dots, j_{2k}=-\infty}^\infty Z_{j_1\dots j_{2k}}\sum\limits_{\alpha=1}^{2k}\,\mu_{j_1\ell_1}\dots\mu_{j_{\alpha-1}\ell_{\alpha-1}}
\frac{d\mu_{j_\alpha\ell_\alpha}}{d\rho}\,\mu_{j_{\alpha+1}\ell_{\alpha+1}}\dots\mu_{j_{2k}\ell_{2k}}=0.
$$
By Proposition \ref{P3.1},
$$
\frac{d\mu_{j_\alpha\ell_\alpha}}{d\rho}=\sum\limits_{p=-\infty}^\infty d_{j_\alpha p}\,\mu_{p\ell_\alpha},
$$
where the matrix $D=(d_{nk})$ is defined by (\ref{4.40}). Substitute this expression into the previous equation
$$
\sum\limits_{j_1,\dots, j_{2k}=-\infty}^\infty Z_{j_1\dots j_{2k}}\sum\limits_{\alpha=1}^{2k}\sum\limits_{p=-\infty}^\infty\,\mu_{j_1\ell_1}\dots\mu_{j_{\alpha-1}\ell_{\alpha-1}}
d_{j_\alpha p}\mu_{p\ell_\alpha}\,\mu_{j_{\alpha+1}\ell_{\alpha+1}}\dots\mu_{j_{2k}\ell_{2k}}=0.
$$
After transposing the summation indices $j_\alpha$ and $p$, this can be written in the form (again, the change of the summation order can be easily justified)
$$
\sum\limits_{j_1,\dots ,j_{2k}=-\infty}^\infty\left(\sum\limits_{\alpha=1}^{2k}\sum\limits_{p=-\infty}^\infty d_{pj_\alpha} Z_{j_1\dots j_{\alpha-1}pj_{\alpha+1}\dots j_{2k}}\right)\mu_{j_1\ell_1}\dots\mu_{j_{2k}\ell_{2k}}=0.
$$
Since the indices $(\ell_1,\dots,\ell_{2k})$ are arbitrary and the matrix $M=(\mu_{j\ell})$ is non-degenerate, this is equivalent to the equation
$$
\sum\limits_{\alpha=1}^{2k}\sum\limits_p d_{pj_\alpha} Z_{j_1\dots j_{\alpha-1}pj_{\alpha+1}\dots j_{2k}}=0.
$$
Substituting value (\ref{4.40}) of $d_{pj_\alpha}$, we obtain the final equation
\begin{equation}
\sum\limits_{\alpha=1}^{2k}\Big((j_\alpha-1) Z_{j_1\dots j_{\alpha-1},j_\alpha+1,j_{\alpha+1}\dots j_{2k}}-(j_\alpha+1) Z_{j_1\dots j_{\alpha-1},j_\alpha-1,j_{\alpha+1}\dots j_{2k}}\Big)=0
                                               \label{6.40}
\end{equation}
which should hold for all indices $(j_1,\dots,j_{2k})$. Conversely, if (\ref{6.40}) was proved, it would imply, together with (\ref{2.4''}), the validity of (\ref{5.-2}) for an arbitrary function $a$.

Equation (\ref{6.40}) can be simplified. The simplification relates to the Lee algebra of the group $G$.

\bigskip

Let us recall that we consider $G$ as a group of diffeomorphisms of the unit circle ${\mathbb S}=\{e^{i\theta}\}$. Therefore the Lee algebra $\mathfrak{g}$ of group $G$ coincides with a three-dimensional space of vector fields on ${\mathbb S}$. As one can easily see,  three vector fields
$$
X_0=\frac{\partial}{\partial\theta},\quad X_1=\cos\theta\frac{\partial}{\partial\theta},\quad X_2=\sin\theta\frac{\partial}{\partial\theta}.
$$
constitute the basis of $\mathfrak{g}$. The Lee product is expressed in the basis by the formulas
\begin{equation}
[X_0,X_1]=-X_2,\quad [X_0,X_2]=X_1,\quad [X_1,X_2]=X_0.
                                               \label{6.-2}
\end{equation}

The group $G$ acts on $C^\infty({\mathbb S})$ transforming a function $a$ to a conformally equivalent function as explained at the beginning of the current section.
Therefore the Lee algebra $\mathfrak{g}$ acts on $C^\infty({\mathbb S})$ too: a vector $A\in\mathfrak{g}$ is considered as a linear operator $A:C^\infty({\mathbb S})\rightarrow C^\infty({\mathbb S})$.  We are going to express the latter action in terms of Fourier coefficients.

We start with the subgroup of rotations $R\subset G$. A rotation acts by the formula $(a R_\alpha)(\theta)=a(e^{i\alpha\theta})$ which implies $(\widehat{a R_\alpha})_n=e^{in\alpha}{\hat a}_n$. Differentiating this equality with respect to $\alpha$ at $\alpha=0$, we obtain $\left.\frac{d}{d\alpha}\right|_{\alpha=0}(\widehat{a R_\alpha})_n=in{\hat a}_n$. We have thus found the first element of $\mathfrak{g}$:
\begin{equation}
(\widehat{Ca})_n=in{\hat a}_n.
                                               \label{6.-1}
\end{equation}

We have already found the element of $\mathfrak{g}$ corresponding the one-dimensional subgroup $T\subset G$. This is the operator $D:C^\infty({\mathbb S})\rightarrow C^\infty({\mathbb S})$ participating in Proposition \ref{P3.1}. By (\ref{4.40}), this operator acts in terms of Fourier coefficients as follows:
\begin{equation}
(\widehat{Da})_{n}=(n-2){\hat a}_{n-1}-(n+2){\hat a}_{n+1}.
                                               \label{6.4}
\end{equation}

To complete $(C,D)$ to a basis of $\mathfrak{g}$, we just evaluate the commutator of (\ref{6.-1}) and (\ref{6.4})
\begin{equation}
E=[C,D],\quad (\widehat{Ea})_{n}=-i\big[(n-2){\hat a}_{n-1}+(n+2){\hat a}_{n+1}\big].
                                               \label{6.0}
\end{equation}

The algebra product is expressed in the basic $(C,D,E)$ by the formulas
\begin{equation}
[C,D]=E,\quad [C,E]=-D,\quad [D,E]=-4C.
                                               \label{6.Lp}
\end{equation}
Formulas (\ref{6.-2}) and (\ref{6.Lp}) are equivalent, this is seen from the basis change
$$
C=X_0,\quad D=2X_2,\quad E=2X_1.
$$

We emphasize that $\mathfrak{g}$ is a real Lee algebra. In particular, the operators $C,D,E$ transform real functions again to real functions. Let ${\mathfrak{g}}_{\mathbb C}$ be the comlexification of $\mathfrak{g}$. The operators
$$
D_0=-iC,\quad D_-=\frac{1}{2}(D+iE),\quad D_+=\frac{1}{2}(-D+iE)
$$
constitute the basis of ${\mathfrak{g}}_{\mathbb C}$. In terms of Fourier coefficients, these operators are defined by the formulas
\begin{equation}
(\widehat{D_0a})_n=n{\hat a}_n,\quad (\widehat{D_-a})_{n}=(n-2){\hat a}_{n-1},\quad (\widehat{D_+a})_{n}=(n+2){\hat a}_{n+1}.
                                               \label{6.pm}
\end{equation}
The Lee product is expressed in this basis as
\begin{equation}
[D_0,D_-]=-D_-,\quad [D_0,D_+]=D_+,\quad [D_-,D_+]=2D_0.
                                               \label{6.L}
\end{equation}

\bigskip

Equation (\ref{6.40}) was actually obtained by differentiating the equality
$$
Z_k(a\Phi_\rho)=Z_k(a\,e^{tD})=Z_k(a)\quad (\tanh t=\rho)
$$
with respect to $t$. Repeating the same arguments for the equation
$$
Z_k(a\,e^{tE})=Z_k(a),
$$
we obtain
\begin{equation}
\sum\limits_{\alpha=1}^{2k}\Big((j_\alpha-1) Z_{j_1\dots j_{\alpha-1},j_\alpha+1,j_{\alpha+1}\dots j_{2k}}+(j_\alpha+1) Z_{j_1\dots j_{\alpha-1},j_\alpha-1,j_{\alpha+1}\dots j_{2k}}\Big)=0.
                                               \label{6.41}
\end{equation}
Taking the sum and difference of (\ref{6.40}) and (\ref{6.41}), we obtain the pair of simpler equations
\begin{equation}
\sum\limits_{\alpha=1}^{2k}(j_\alpha-1) Z_{j_1\dots j_{\alpha-1},j_\alpha+1,j_{\alpha+1}\dots j_{2k}}=0,
                                               \label{6.42}
\end{equation}
\begin{equation}
\sum\limits_{\alpha=1}^{2k}(j_\alpha+1) Z_{j_1\dots j_{\alpha-1},j_\alpha-1,j_{\alpha+1}\dots j_{2k}}=0.
                                               \label{6.43}
\end{equation}
Of course, equations (\ref{6.42}) and (\ref{6.43}) correspond to the operators $D_+,D_-\in{\mathfrak{g}}_{\mathbb C}$ as well as equations (\ref{6.40}) and (\ref{6.41}) correspond to $D,E\in{\mathfrak{g}}$.

We observe that equations (\ref{6.42}) and (\ref{6.43}) are equivalent modulo the evenness condition (\ref{2.4''}). Indeed, if we change signs of all indices $(j_1,\dots,j_{2k})$ in (\ref{6.43}) and use property (\ref{2.4''}), then we get (\ref{6.42}). Therefore equation (\ref{6.43}) can be eliminated from our considerations. Finally, (\ref{6.42}) is trivially valid in the case of
$j_1+\dots+j_{2k}\neq -1$ since, according to the definition in Section 2, $Z_{j_1\dots j_{2k}}=0$ for $j_1+\dots+j_{2k}\neq 0$. Thus, relations (\ref{6.42})--(\ref{6.43}) are reduced to the equation
\begin{equation}
\sum\limits_{\alpha=1}^{2k}(j_\alpha-1)Z_{j_1\dots j_{\alpha-1},j_\alpha+1,j_{\alpha+1},\dots j_{2k}}=0\quad (j_1+\dots +j_{2k}=-1).
                                               \label{6.12}
\end{equation}

{\bf Remark.} We have proved that equation (\ref{6.12}), together with the evenness condition (\ref{2.4''}), is equivalent to the conformal invariance (\ref{5.-2}) of the zeta-invariant $Z_k(a)$.
We emphasize that our proof is based on using Theorem \ref{T2.1}. Can equation (\ref{6.12}) be proved without using Steklov spectra, i.e., on the base of the definition (\ref{2.4}) and (\ref{2.3.3}) of the coefficients $Z_{j_1\dots j_{2k}}$? So far, we cannot find such a direct proof for a general $k$. The only exceptions are the cases of $k=1,2$. In the case of $k=1$, (\ref{6.12}) can be easily derived from Edward's formula (\ref{2.12'}). In the case of $k=2$, (\ref{6.12}) can be derived from explicit formulas for coefficients $Z_{ijk\ell}$ given by Theorem \ref{T7.1} below.

\section{Explicit formula for coefficients of the second zeta-invariant}

Edward's formula (\ref{2.12'}) means that coefficients of the quadratic form $Z_2(a)=\sum_i Z_{i,-i}{\hat a}_i{\hat a}_{-i}$ are expressed by a third degree piece-wise polynomial function in $i$,
$$
Z_{i,-i}=\left\{\begin{array}{ll}
\frac{1}{3}(i^3-i)\quad&\mbox{if}\quad i\geq 0,\\
\frac{1}{3}(-i^3+i)\quad&\mbox{if}\quad i\leq 0.
\end{array}\right.
$$
We emphasize also the following interesting circumstance: both polynomials participating in the formula are odd in $i$ while the coefficient $Z_{i,-i}$ is even. A similar statement on the second zeta-invariant sounds as follows:

\begin{theorem} \label{T7.1}
Coefficients of the 4-form
$$
Z_2(a)=\sum\limits_{i,l,k,\ell}Z_{ijk\ell}\, {\hat a}_i{\hat a}_j{\hat a}_k{\hat a}_\ell
$$
are completely determined by the following:

{\rm (1)} $Z_{ijk\ell}=0$ for $i+j+k+\ell\neq 0$;

{\rm (2)} $Z_{ijk\ell}$ are symmetric in $(i,j,k,\ell)$ and even: $Z_{-i,-j,-k,-\ell}=Z_{i,j,k,\ell}$;

{\rm (3)} $Z_{ijk,-i-j-k}$ is expressed through $(i,j,k)$ by the formula
\begin{equation}
Z_{ijk,-i\!-\!j\!-\!k} =\left\{
\begin{array}{llll}
P_{1}(i,j,k)&  \mbox{if}& i\geq 0,\ j\geq0,\ k\geq0;\\
P_{2}(i,j,k)&  \mbox{if}& i\leq0,\ j\geq0,\ k\geq0,\ i+j\leq 0,\ i+k\leq 0,\ i+j+k\geq 0
\end{array}
\right.
                                      \label{7.81}
\end{equation}
where polynomials $P_1$ and $P_2$ are defined by the equalities
\begin{equation}
P_{1}(i,j,k)
=\frac{1}{15}\sigma_{(ijk)}\big(3i^5+15i^4j+10i^3j^2+10i^3jk
-5i^3-25i^2j-10ijk+2i\big),
                              \label{7.82}
\end{equation}
\begin{equation}
\begin{aligned}
P_{2}(i,j,k)=\frac{1}{45}\sigma_{(jk)}\Big(&5i^5+25i^4j+10i^3j^2+20i^3jk-10i^2j^3-15ij^4-20ij^3k\\
&-4j^5-5j^4k+10j^3k^2
-5i^3-15i^2j+5ij^2-5j^2k+4j\Big).
\end{aligned}
                              \label{7.82'}
\end{equation}
Here $\sigma_{(ijk)}$ ($\sigma_{(jk)}$) stands for the symmetrization in indices $(i,j,k)$ (in indices $(j,k)$).
\end{theorem}

We emphasize that $P_1$ and $P_2$ are fifth degree polynomials and they are odd, i.e.,
$ P_r(-i,-j,-k)=-P_r(i,j,k)\ (r=1,2)$. Besides this, the polynomials possess interesting positiveness and divisibility properties since $3Z_{ijk\ell}$ is a non-negative even integer.
Most probably, the same statement is true for higher order invariants $Z_k$: coefficients $Z_{j_1\dots j_{2k\!-\!1},-j_1\!-\!\dots\!-\!j_{2k\!-\!1}}$ of $2k$-form (\ref{2.3.1}) are expressed by a piece-wise polynomial function in $(j_1,\dots,j_{2k-1})$ represented by odd polynomials of degree $2k+1$. Unfortunately, for $k>2$, these polynomials are too complicated to be really useful.

To prove Theorem \ref{T7.1}, we need the following

\begin{lemma} \label{L7.3}
Modulo statements (1) and (2) of Theorem \ref{T7.1}, all coefficients $Z_{ijk\ell}$ are completely determined by the following:

values of  $Z_{ijk,-i\!-\!j\!-\!k}\ \mbox{for}\ i\geq 0,\ j\geq0,\ k\geq0 \quad (\mbox{\rm \bf Case 1})$;

values of $Z_{ijk,-i\!-\!j\!-\!k}\ \mbox{for}\ i\leq0,\ j\geq0,\ k\geq0,\ i+j\leq 0,\ i+k\leq 0,\ i+j+k\geq 0\ (\mbox{\rm \bf Case 2})$.
\end{lemma}

\begin{proof}
We consider the set of all quadruples of integers $(i,j,k,\ell)$ satisfying $i+j+k+\ell=0$. The set is the union of the following two subsets:

(a) the set of all quadruples $(i,j,k,\ell)$ such that three elements of the quadruple have the same sign (on assuming that 0 has both signs);

(b) the set of all quadruples $(i,j,k,\ell)$ such that two elements of the quadruple are non-negative and two other elements are non-positive.

In virtue of statements (1) and (2) of the theorem, we can permute elements of the quadruple and can change their signs simultaneously. In case (a), we use this ambiguity to get
$i\geq0,\ j\geq0,\ k\geq0$. This is exactly case 1 of the lemma.

In case (b), we use the ambiguity to get
\begin{equation}
i\leq0,\quad |i|=\max\{|i|,|j|,|k|,|\ell|\}.
                              \label{7.83}
\end{equation}
Now, two elements of the triple $(j,k,\ell)$ are non-negative and one element is non-positive. We permute elements of the triple so that
\begin{equation}
j\geq0,\quad k\geq0,\quad \ell\leq0.
                              \label{7.84}
\end{equation}
A simple arithmetic analysis shows that the union of conditions (\ref{7.83}) and (\ref{7.84}) is equivalent to
\begin{equation}
i\leq0,\ j\geq0,\ k\geq0,\ i+j\leq 0,\ i+k\leq 0,\ i+j+k\geq 0,\quad \ell=-(i+j+k).
                              \label{7.85}
\end{equation}
This is exactly case 2 of the lemma.
\end{proof}

\begin{proof}[Proof of Theorem \ref{T7.1}.]
The detailed proof involves a number of routine but rather cumbersome calculations with polynomials (multiplication of two polynomials, grouping similar terms in a polynomial). We implemented such calculations on a computer with the help of the symbolic calculations package MAPLE. The calculations are omitted in the proof presented below.

We introduce the notation
\begin{equation}
\{x\}=|x|-x=\left\{\begin{array}{ll}
0\quad&\mbox{for}\quad x\geq 0,\\
-2x&\mbox{for}\quad x< 0.
\end{array}\right.
                                               \label{6.15'}
\end{equation}
Let us fix $(i,j,k,\ell)$ satisfying $i+j+k+\ell=0$ and define the polynomial
\begin{equation}
f(n)=n(n+i)(n+i+j)(n+i+j+k).
                                               \label{7.0}
\end{equation}
Then formula (\ref{2.4}) can be written as
\begin{equation}
N_{ijk\ell}=\sum\limits_n\{f(n)\}.
                                               \label{7.1}
\end{equation}
Roots of the polynomial $f$ are elements of the set $\{0,-i,-i-j,-i-j-k\}$. Let $(r_1,r_2,r_3,r_4)$ be the sequence of the roots ordered by their values, i.e.,
$$
\{r_1,r_2,r_3,r_4\}=\{0,-i,-i-j,-i-j-k\},\quad r_1\leq r_2\leq r_3\leq r_4.
$$
Formula (\ref{7.1}) can be rewritten as
\begin{equation}
N_{ijk\ell}=-2\sum\limits_{n=r_1}^{r_2}f(n)-2\sum\limits_{n=r_3}^{r_4}f(n).
                                               \label{7.2}
\end{equation}
Transform (\ref{7.0}) to the form
\begin{equation}
f(n)=n^4+\alpha_1 n^3+\alpha_2 n^2+\alpha_3 n,
                                               \label{7.3}
\end{equation}
where
\begin{equation}
\alpha_1=3i+2j+k,\quad\alpha_2=3i^2+4ij+2ik+j^2+jk,\quad\alpha_3=i^3+2i^2j+i^2k+ij^2+ijk.
                                               \label{7.4}
\end{equation}

Now, we are going to evaluate the first sum on the right-hand side of (\ref{7.2}).
Let us first assume that $0\leq r_1\leq r_2$. Then
$$
\sum\limits_{n=r_1}^{r_2}f(n)=\sum\limits_{n=0}^{r_2}f(n)-\sum\limits_{n=0}^{r_1}f(n).
$$
We have used that $f(r_1)=0$. Substitute  value (\ref{7.3}) into the last formula
\begin{equation}
\sum\limits_{n=r_1}^{r_2}f(n)=\sum\limits_{n=0}^{r_2}n^4-\sum\limits_{n=0}^{r_1}n^4
+\alpha_1\Big(\sum\limits_{n=0}^{r_2}n^3-\sum\limits_{n=0}^{r_1}n^3\Big)
+\alpha_2\Big(\sum\limits_{n=0}^{r_2}n^2-\sum\limits_{n=0}^{r_1}n^2\Big)
+\alpha_3\Big(\sum\limits_{n=0}^{r_2}n-\sum\limits_{n=0}^{r_1}n\Big).
                                               \label{7.5}
\end{equation}
Using (\ref{2.15}) and the similar formulas \cite[Section 4.1.1]{PBM}
$$
\sum\limits_{n=0}^{r}n^3=\frac{1}{4}r^2(r+1)^2,\quad \sum\limits_{n=0}^{r}n^4=\frac{1}{30}r(r+1)(2r+1)(3r^2+3r-1),
$$
we obtain from (\ref{7.5})
\begin{equation}
\sum\limits_{n=r_1}^{r_2}f(n)=\varphi(r_2)-\varphi(r_1),
                                               \label{7.10}
\end{equation}
where
\begin{equation}
\varphi(r)=r(r+1)\Big[\frac{1}{30}(2r+1)(3r^2+3r-1)+\frac{\alpha_1}{4}r(r+1)+\frac{\alpha_2}{6}(2r+1)+\frac{\alpha_3}{2}\Big]
                                               \label{7.9}
\end{equation}
is the discrete antiderivative of $f(n)$. As one can easily see, (\ref{7.10}) holds also in two other cases when either $r_1\leq 0\leq r_2$ or $r_1\leq r_2\leq 0$.
Thus, (\ref{7.10}) is an universal formula, i.e., it is valid for all values of the roots $r_1\leq r_2\leq r_3\leq r_4$. Of course, a similar formula holds for the second sum on the right-hand side of (\ref{7.2}).

We substitute (\ref{7.10}) and the similar expression  for the second sum into (\ref{7.2})
\begin{equation}
N_{ijk\ell}=2\big(\varphi(r_1)-\varphi(r_2)+\varphi(r_3)-\varphi(r_4)\big).
                                               \label{7.11}
\end{equation}

Next, we are going to symmetrize formula (\ref{7.11}) in the indices
$(i,j,k)$ in order to obtain a formula for $Z_{ijk\ell}\
(i+j+k+\ell=0)$. To this end we use formula (\ref{2.3.4}) that
is reproduced here in the form
\begin{equation}
3Z_{ijk\ell}=\frac{1}{2}\big(N_{ijk\ell}+N_{ikj\ell}+N_{jik\ell}+N_{jki\ell}+N_{kij\ell}+N_{kji\ell}\big)\quad
(i+j+k+\ell=0).
                                               \label{7.12}
\end{equation}

The main difficulty
relates to the following circumstance: the roots $(r_1,r_2,r_3,r_4)$ must be expressed
through the indices $(i,j,k)$. This expression has different forms
in different cases. In virtue of Lemma \ref{L7.3}, it suffices to consider two cases mentioned in the lemma.

{\bf Case 1.} Assume that $i\geq 0,\ j\geq 0,\ k\geq 0$. Then
\begin{equation}
r_1=-i-j-k,\quad r_2=-i-j,\quad r_3=-i,\quad r_4=0.
                                               \label{7.14}
\end{equation}
Substituting these values into (\ref{7.9}), we express $\varphi(r_m)\ (1\leq m\leq4)$ through $(i,j,k)$. Then we substitute the expressions for $\varphi(r_m)$ into (\ref{7.11}) to obtain a formula expressing $N_{ijk\ell}$ as a fifth degree polynomial in the variables $(i,j,k)$.
Finally, we symmetrize the polynomial, i.e., substitute it into (\ref{7.12}). It is important to note that, in case 1, we do not
need to take care of the order of roots $(r_1,r_2,r_3,r_4)$ in
different terms on the right-hand side of (\ref{7.12}); the order
will be automatically changed in the right way.
For example, the
second term $N_{ikj\ell}$ is obtained from the first term
by transposition of indices $(j,k)$. For this term
$$
r_1=-i-j-k,\quad r_2=-i-k,\quad r_3=-i,\quad r_4=0.
$$
These formulas are obtained from (\ref{7.14}) by the same
transposition.
Finally, we arrive to the equality $Z_{ijk,-i\!-\!j\!-\!k} =P_1(i,j,k)$, where $P_1$ is defined by (\ref{7.82}).

{\bf Case 2.} Assume that $i\leq 0,\ j\geq 0,\ k\geq 0, \ i+j\leq 0,\ i+k\leq 0,\ i+j+k\geq 0$. In this case, the roots $(r_1,r_2,r_3,r_4)$ take different values for different terms on the right-hand side of (\ref{7.12}). Namely,
$$
\begin{aligned}
r_1&=-i-j-k,\ r_2=0,\ r_3=-i-j,\ r_4=-i\quad&\mbox{for}\quad N_{ijk\ell};\\
r_1&=-j,\ r_2=-i-j-k,\ r_3=0,\ r_4=-i-j\quad&\mbox{for}\quad N_{jik\ell};\\
r_1&=-j-k,\ r_2=-j,\ r_3=-i-j-k,\ r_4=0\quad&\mbox{for}\quad N_{jki\ell}.
\end{aligned}
$$
Similar formulas for other three terms on the right-hand side of (\ref{7.12}) are obtained by transposing the indices $(j,k)$ here.
Using these values, we repeat our calculations and arrive to the equality $Z_{ijk,-i\!-\!j\!-\!k} =P_2(i,j,k)$, where $P_2$ is defined by (\ref{7.82'}).
\end{proof}

\section{Some open questions}

Let us recall our main problem posed in Section 1: given a positive function $b\in C^\infty({\mathbb S})$, one has to find all positive functions $a\in C^\infty({\mathbb S})$ satisfying $\mbox{Sp}\,(a\Lambda_e)=\mbox{Sp}\,(b\Lambda_e)$. Zeta-invariants allow us to write down the infinite system of equations
\begin{equation}
Z_k(a)=b_k\quad(k=1,2,\dots)
                          \label{14.-1}
\end{equation}
in the Fourier coefficients of the function $a$, where $b_k=Z_k(b)$. This reduces our problem to the algebraic problem of studying system (\ref{14.-1}) (if equations with formal power series are considered as algebraic equations).
The principle question on zeta-invariants is the following one: are the invariants $Z_k(a)\ (k=1,2,\dots)$ independent of each other, i.e., does system (\ref{14.-1})
give us infinitely many conditions on the Fourier coefficients of a function $a$? We believe this is true but cannot prove so far.
The first and second zeta-invariants are independent. Indeed, by (\ref{2.13}), $Z_1(a)$ is independent of $({\hat a}_0,{\hat a}_{\pm 1})$. On the other hand, the 4-form $Z_2(a)$ contains summands of the form ${\hat a}_0^2{\hat a}_ka_{-k}$ with non-zero coefficients, as one can see with the help of Theorem \ref{T7.1}.

\bigskip

As one can easily see, $Z_k(a)=0\ (k=1,2,\dots)$ for every function $a$ belonging to the three-dimensional subspace
$$
L=\{a\in C^\infty({\mathbb S})\mid a(\theta)={\hat a}_{0}+{\hat a}_{1}e^{i\theta}+{\hat a}_{-1}e^{-i\theta}\}
$$
of the space $C^\infty({\mathbb S})$. Indeed, as is seen from (\ref{2.4}), $N_{j_1\dots j_{2k}}=0$ if each of indices $(j_1,\dots, j_{2k})$ is equal either to zero or to $\pm 1$.
The converse statement is true in the case of $k=1$ for real functions: if $Z_1(a)=0$ for a real function $a\in C^\infty({\mathbb S})$, then $a\in L$. This is seen from Edward's formula (\ref{2.13}) that takes the following form in the case of a real function $a$:
\begin{equation}
Z_1(a)=\frac{2}{3}\sum\limits_{n=2}^\infty(n^3-n)\, |{\hat a}_n|^2.
                                               \label{14.0}
\end{equation}
How does the set of all (real) functions $a\in C^\infty({\mathbb S})$ satisfying $Z_k(a)=0$ for $k=2,3,\dots$ look like, can it be essentially different of $L$?

\bigskip

As is seen from (\ref{14.0}), the estimate
$$
Z_1(a)\geq c_1\sum\limits_{n\geq2}n^3|{\hat a}_n|^2
$$
with some universal constant $c_1>0$ holds
for every real function $a\in C^\infty(\gamma)$.

\begin{problem} \label{P14.1}
Does the inequality
\begin{equation}
Z_k(a)\geq c_k\sum\limits_{n\geq2}n^{2k+1}|{\hat a}_n|^{2k}
                          \label{14.1}
\end{equation}
hold for every real function $a\in C^\infty({\mathbb S})$ and for every $k=2,3,\dots$, where the coefficient $c_k>0$ depends on $k$ only? If the answer is "no", the same question can be asked for positive functions $a$.
\end{problem}

So far, even the inequality $Z_k(a)\geq0\ (k=2,3,\dots)$ remains unproved for a real  $a$. Along with (\ref{14.0}), the following ``naive'' argument can be mentioned to justify the inequality:
by (\ref{2.1}), $Z_k(a)=\zeta_a(-2k)=\mbox{Tr}(B^2)$ for some self-dual operator $B$.
We checked numerically the inequality $Z_2(a)\geq0$ for many functions that were chosen by a more or less random choice of Fourier coefficients satisfying $\overline{{\hat a}_{n}}={\hat a}_{-n}$ and ${\hat a}_{n}=0$ for $|n|>n_0$ with some $n_0$. The inequality holds in all considered cases.

\bigskip

Compactness theorems of the following kind are popular in Spectral Geometry (see \cite{BPP} and references there): a family of Riemannian manifolds (satisfying some additional conditions) whose Laplacians have the same eigenvalue spectrum is (pre)compact in an appropriate topology. Let us discuss one of possible compactness theorems for the Steklov spectrum. Of course, the conformal equivalence should be taken into account since the conformal group is non-compact.

Let us remind that the Hilbert space $H^s({\mathbb S})$ is the completion of $C^\infty({\mathbb S})$ with respect to the norm
$$
\|a\|_{H^s({\mathbb S})}^2=\sum\limits_n(1+|n|^{2s})|{\hat a}_n|^2.
$$
In our opinion, $\|a\|_{H^{3/2}({\mathbb S})}$ is the most appropriate norm for studying the compactness. Indeed, as is seen from (\ref{2.13}),
$$
\|a\|_{H^{3/2}({\mathbb S})}^2\sim |{\hat a}_{0}|^2+|{\hat a}_{1}|^2+Z_1(a)
$$
for a real function $a$.

Let us consider a sequence of positive functions $a^\nu\in C^\infty({\mathbb S})\ (\nu=1,2,\dots)$ such that the Steklov spectrum $\mbox{\rm Sp}\,(a^\nu\Lambda_e)$ is independent of $\nu$.
As is seen from (\ref{14.0}), the estimate
\begin{equation}
|{\hat a}^\nu_n|\leq C|n|^{-3/2}\quad\mbox{при}\quad |n|\geq2
                          \label{14.2}
\end{equation}
holds with some constant $C$ independent of $\nu$. Therefore the sequence $|{\hat a}^\nu_n|\ (\nu=1,2,\dots)$ is bounded for every $|n|\geq2$. The positiveness of $a^\nu$ implies the inequality $|{\hat a}^\nu_1|\leq{\hat a}^\nu_0$. Thus, the only obstruction to the boundedness of the sequence of norms $\|a^\nu\|_{H^{3/2}({\mathbb S})}\ (\nu=1,2,\dots)$ is the possible unboundedness of the sequence ${\hat a}^\nu_0\ (\nu=1,2,\dots)$. The latter sequence can be unbounded as easy examples show. We try to overrun the obstruction by replacing each function $a^\nu$ with some conformally equivalent function. In this way we arrive to the statement:

\begin{theorem} \label{Tcomp}
Let $a^\nu\in C^\infty({\mathbb S})\ (\nu=1,2,\dots)$ be a sequence of functions uniformly bounded from below by some positive constant
$$
a^\nu(\theta)\geq c>0.
$$
Assume the Steklov spectrum $\mbox{\rm Sp}\,(a^\nu\Lambda_e)$ to be independent of $\nu$. Then there exists a subsequence $a^{\nu_k}$ such that every function $a^{\nu_k}$ is conformally equivalent to some function $b^{k}\in C^\infty({\mathbb S})$ and the sequence of norms $\|b^{k}\|_{H^{3/2}({\mathbb S})}$ is bounded. Hence, for every $s<3/2$, the sequence $b^{k}$ contains a subsequence converging in $H^s({\mathbb S})$.
\end{theorem}

The theorem is not proved yet. In our approach to the proof, the main difficulty relates to estimate (\ref{14.2}). We can prove the theorem if, instead of (\ref{14.2}), the following stronger estimate holds:
\begin{equation}
|{\hat a}^\nu_n|\leq C|n|^{-3/2-\varepsilon}\quad\mbox{for}\quad |n|\geq2\quad(C\ \mbox{is independent of}\ \nu),
                          \label{14.3}
\end{equation}
where $\varepsilon>0$ can be arbitrary.

The possibility of proving estimates like (\ref{14.3}) closely relates to Problem \ref{P14.1} (more precisely, to the inequality  $Z_2(a)\geq0$). Let us briefly explain the relation. For the sake of simplicity, let us consider a real function $a$ with the real first Fourier coefficient ${\hat a}_1={\hat a}_{-1}$. As one can easily see, $Z_2(a)$ is a second degree polynomial in the variables $({\hat a}_0,{\hat a}_1)$. Indeed, $Z_{ijk\ell}=0$ if three elements of the quadruple $(i,j,k,\ell)$ belong to the set $\{0,1,-1\}$; this easily follows from definition (\ref{2.4}).
In particular, setting ${\hat a}_1=\kappa{\hat a}_0$, we can write the second zeta-invariant as a quadratic trinomial in the variable ${\hat a}_0$
\begin{equation}
Z_2(a)=A_\kappa({\hat a}_2,{\hat a}_3,\dots){\hat a}_0^2+2B_\kappa({\hat a}_2,{\hat a}_3,\dots){\hat a}_0+N({\hat a}_2,{\hat a}_3,\dots).
                              \label{14.4}
\end{equation}
The first coefficient of the trinomial can be easily found with the help of Theorem \ref{T7.1}
\begin{equation}
\begin{aligned}
\frac{5}{4}A_\kappa({\hat a}_2,{\hat a}_3,\dots)&=
(1+2\kappa^2)\sum\limits_{n\geq2}n(n^2-1)(n^2-2/3)|{\hat a}_n|^2\\
&+2\kappa\sum\limits_{n\geq2}n(n^2-1)(n+2)(n+1/2)({\hat a}_n\overline{{\hat a}_{n+1}}+{\hat a}_{n+1}\overline{{\hat a}_{n}})\\
&+\kappa^2\sum\limits_{n\geq2}n(n^2-1)(n+2)(n+3)({\hat a}_n\overline{{\hat a}_{n+2}}+{\hat a}_{n+2}\overline{{\hat a}_{n}}).
\end{aligned}
                              \label{14.5}
\end{equation}
Assuming the inequality $Z_2(a)\geq0$ to be valid for every real function $a$, we write (\ref{14.4}) in the form
\begin{equation}
A_\kappa({\hat a}_2,{\hat a}_3,\dots){\hat a}_0^2+2B_\kappa({\hat a}_2,{\hat a}_3,\dots){\hat a}_0+N({\hat a}_2,{\hat a}_3,\dots)=Z_2(a)\geq0.
                              \label{14.6}
\end{equation}
Since ${\hat a}_0$ is arbitrary, this implies the positive definiteness of the Hermitian form $A_\kappa({\hat a}_2,{\hat a}_3,\dots)$ for every $\kappa$. Moreover, (\ref{14.6}) can be rewritten in the form
$$
\begin{aligned}
\Big[\Big(A_\kappa({\hat a}_2,{\hat a}_3,\dots)-\delta\sum\limits_{n\geq2}|n|^5\,|{\hat a}_n|^2\Big){\hat a}_0^2&+2B_\kappa({\hat a}_2,{\hat a}_3,\dots){\hat a}_0+N({\hat a}_2,{\hat a}_3,\dots)\Big]\\
&+\delta\sum\limits_{n\geq2}|n|^5\,|{\hat a}_n|^2=Z_2(a).
\end{aligned}
$$
As is seen from (\ref{14.5}), the structure of the expression in the brackets is very similar to that of the left-hand side of (\ref{14.6}). If we had proven the inequality $Z_2(a)\geq0$ for a real $a$, then, probably, similar arguments would allow us to prove the non-negativeness of the expression in the brackets, at least for a sufficiently small $\delta>0$. If so, the last inequality gives under assumptions of Theorem \ref{Tcomp}
$$
\delta\sum\limits_{n\geq2}|n|^5\,|{\hat a}^\nu_n|^2\leq Z_2(a^\nu)=\mbox{const}
$$
This implies estimate (\ref{14.3}) with $\varepsilon=1$.
Concluding the discussion, we repeat again: compactness theorems for the Steklov spectrum are closely related to Problem \ref{P14.1}.


\end{document}